\documentclass[a4paper]{amsart}
\usepackage{amssymb, enumitem, hyperref, aliascnt, graphicx}
\usepackage{tikz-cd}
\usepackage[capitalise, nameinlink, noabbrev, nosort]{cleveref} 

\theoremstyle{plain}
\newtheorem{lma}{Lemma}[section]
\crefname{lma}{Lemma}{Lemmata}
\newtheorem{thm}[lma]{Theorem}
\crefname{thm}{Theorem}{Theorems}
\newtheorem{cor}[lma]{Corollary}
\crefname{cor}{Corollary}{Corollaries}
\newtheorem{prp}[lma]{Proposition}
\crefname{prp}{Proposition}{Propositions}

\theoremstyle{definition}
\newtheorem{pgr}[lma]{}
\crefname{pgr}{Paragraph}{Paragraphs}
\newtheorem{dfn}[lma]{Definition}
\crefname{dfn}{Definition}{Definitions}

\theoremstyle{remark}
\newtheorem{rmk}[lma]{Remark}
\crefname{rmk}{Remark}{Remarks}
\newtheorem{exa}[lma]{Example}
\crefname{exa}{Example}{Examples}
\newtheorem{qst}[lma]{Question}
\crefname{qst}{Question}{Questions}

\theoremstyle{plain}

\newcounter{IntroCount}

\newtheorem{dfnIntro}[IntroCount]{Definition}
\crefname{dfnIntro}{Definition}{Definitions}
\newtheorem{thmIntro}[IntroCount]{Theorem}
\crefname{thmIntro}{Theorem}{Theorems}
\newtheorem{exaIntro}[IntroCount]{Example}
\crefname{exaIntro}{Example}{Examples}

\def\today{\number\day\space\ifcase\month\or   January\or February\or
   March\or April\or May\or June\or   July\or August\or September\or
   October\or November\or December\fi\   \number\year}

\newcommand{\andSep}{\,\,\,\text{ and }\,\,\,}
\newcommand{\axiomO}[1]{(O#1)}
\newcommand{\CuSgp}{$\mathrm{Cu}$-sem\-i\-group}
\newcommand{\CuMor}{$\mathrm{Cu}$-mor\-phism}
\newcommand{\calZ}{\mathcal{Z}}
\newcommand{\calC}{\mathcal{C}}
\newcommand{\calD}{\mathcal{D}}
\newcommand{\calE}{\mathcal{E}}
\newcommand{\calF}{\mathcal{F}}
\newcommand{\calP}{\mathcal{P}}
\newcommand{\calQ}{\mathcal{Q}}
\newcommand{\NN}{{\mathbb{N}}}
\newcommand{\Cpct}{{\mathcal{K}}}
\newcommand{\ca}{$C^*$-algebra}

\DeclareMathOperator{\Sep}{Sep}
\DeclareMathOperator{\Cu}{Cu}

\title{Extensions of pure C*-algebras}
\date{\today}

\author{Francesc Perera}
\address{
F.~Perera, 
Departament de Matem\`{a}tiques,
Universitat Aut\`{o}noma de Barcelona,
\linebreak 08193 Bellaterra, Barcelona, Spain, and
Centre de Recerca Matem\`atica, Edifici Cc, Campus de Bellaterra,  08193 Cerdanyola del Vall\`es, Barcelona, Spain}
\email[]{francesc.perera@uab.cat}
\urladdr{https://mat.uab.cat/web/perera}

\author{Hannes Thiel}
\address{Hannes~Thiel, 
Department of Mathematical Sciences, Chalmers University of Technology and the University of Gothenburg, SE-412 96 Gothenburg, Sweden}
\email{hannes.thiel@chalmers.se}
\urladdr{www.hannesthiel.org}

\author{Eduard Vilalta}
\address{Eduard Vilalta, 
Department of Mathematical Sciences, Chalmers University of Technology and University of Gothenburg, SE-412 96 Gothenburg, Sweden}
\email[]{vilalta@chalmers.se}
\urladdr{www.eduardvilalta.com}

\thanks{
FP and EV were partially supported by MINECO (grants No.\  PID2023-147110NB-I00) and by the Comissionat per Universitats i Recerca de la Ge\-ne\-ralitat de Ca\-ta\-lu\-nya (grant No.\ 2021 SGR 01015). FP was also supported by the Spanish State Research Agency through the Severo Ochoa and María de Maeztu Program for Centers and Units of Excellence in R\&D (CEX2020-001084-M).
HT and EV were partially supported by the Knut and Alice Wallenberg Foundation (KAW 2021.0140).
}

\subjclass[2010]%
{Primary
46L05; % General theory of C*-algebras
Secondary
19K14, % $ K_0$ as an ordered group, traces
46L80, % $K$-theory and operator algebras 
46L85. %Noncommutative topology
}
\date{\today}

\begin{document}

%==========================================================================================
\begin{abstract}
Given a closed ideal $I$ in a \ca{} $A$, we show that $A$ is pure if and only if $I$ and $A/I$ are pure.
More generally, we study permanence of comparison and divisibility properties when passing to extensions.
%As a key ingredient we show that pureness of a \ca{} is determined by its separable subalgebras.

As an application we show that stable multiplier algebras of reduced free group \ca{s} are pure.
\end{abstract}

\maketitle

%==========================================================================================
%==========================================================================================
\section{Introduction}

%==========================================================================================
Pureness is a regularity property for \ca{s} that plays a central role in the modern structure and classification theory of operator algebras. 
Introduced by Winter in the study of nuclear dimension and $\calZ$-stability of simple, nuclear \ca{s} \cite{Win12NuclDimZstable}, it captures two fundamental structural features: 
strict comparison of positive elements and almost divisibility in the Cuntz semigroup. 
These properties mirror the behavior of projections in $\mathrm{II}_1$ factors, and they can be regarded as a Cuntz semigroup-level analog of $\calZ$-stability.

The significance of $\mathcal{Z}$-stability -- meaning tensorial absorption of the Jiang–Su algebra $\mathcal{Z}$ -- was demonstrated by Toms’ example \cite{Tom08ClassificationNuclear} of two simple, nuclear \ca{s} with identical Elliott invariants but distinct regularity properties,  only one being $\mathcal{Z}$-stable. 
This led to a shift in the Elliott classification program, elevating $\mathcal{Z}$-stability to a central hypothesis.
R{\o}rdam \cite{Ror04StableRealRankZ} showed that $\calZ$-stability always implies pureness.
Conversely, the Toms-Winter conjecture \cite{Win18ICM} predicts in particular that simple, nuclear, pure \ca{s} are $\calZ$-stable.
This implication has been verified in several important cases (see, for example, \cite{Win12NuclDimZstable, Sat12arx:TraceSpace, KirRor14CentralSeq, TomWhiWin15ZStableFdBauer, Thi20RksOps}), but remains open in general.

Outside the nuclear setting, $\mathcal{Z}$-stability is typically absent, and pureness emerges as the more relevant notion of regularity.
Many natural classes of \ca{s} are known to be pure even though they are typically not $\calZ$-stable, including von Neumann algebras with no type~I summand, purely infinite \ca{s}, and reduced group \ca{s} of certain non-amenable groups \cite{AmrGaoElaPat24arX:StrCompRedGpCAlgs}. 
Understanding when pureness holds becomes therefore an important problem in \ca{} theory. 
Our main result establishes a fundamental permanence property for pureness:

%==========================================================================================
\begin{thmIntro}[\ref{prp:PureExt}]
	\label{thmA}
Let $I$ be a closed ideal in a \ca{} $A$.
Then $A$ is pure if and only if~$I$ and~$A/I$ are pure.
\end{thmIntro}

%==========================================================================================
This is analogous to the result of Toms and Winter \cite[Theorem~4.3]{TomWin07ssa} that $\calZ$-stability passes to extensions of \ca{s}.

As an application of our main result, we show that stable multiplier algebras of reduced group \ca{s} can be pure, even when the underlying algebra is not $\mathcal{Z}$-stable.
For example, the reduced \ca{} of the free group  $\mathbb{F}_n$ fails to be $\calZ$-stable because its von Neumann envelope $L(\mathbb{F}_n)$ is a non-McDuff $\mathrm{II}_1$~factor, as shown by Ge in \cite[Theorem~3.3]{Ge98ApplFreeEntropy2}. 
This implies that the stabilization $C^*_{\mathrm{red}}(\mathbb{F}_n)\otimes\Cpct$ does not absorb the Jiang–Su algebra either. 
On the other hand, recent work of Amrutam, Gao, Kunnawalkam Elayavalli, and Patchell  \cite{AmrGaoElaPat24arX:StrCompRedGpCAlgs} shows that $C^*_{\mathrm{red}}(\mathbb{F}_n)\otimes\Cpct$ is pure.
Building on work of Kaftal, Ng, and Zhang \cite{KafNgZha17StrCompMultiplier} on the structure of corona algebras, we extend this to the multiplier setting:

%==========================================================================================
\begin{exaIntro}[See \ref{exa:GroupCa}]
For each $n = 2,3,\ldots,\infty$, the stable multiplier algebra $M(C^*_{\mathrm{red}}(\mathbb{F}_n)\otimes\Cpct)$ of the reduced free group \ca{} $C^*_{\mathrm{red}}(\mathbb{F}_n)$ is pure.
\end{exaIntro}

%==========================================================================================
This should be compared with the recent result of Farah and Szab\'{o} \cite{FarSza24arX:CoronaSSA}, which shows that the multiplier algebra of a $\sigma$-unital, $\calZ$-stable \ca{} is separably $\calZ$-stable and therefore pure. 

\medskip

A key ingredient in the proof of \cref{thmA} is the observation that pureness is a \emph{separably determined} property.
This concept, which we formalize in the following definition and develop further in the sequel, provides a general framework for verifying properties of nonseparable algebras by examining their separable subalgebras.

%==========================================================================================
\begin{dfnIntro}[See \ref{dfn:SepDet}, \ref{rmk:SepDet}]
\label{dfn:SepDetIntro}
We say that a property $\calP$ of \ca{s} is \emph{separably determined} if:
\begin{enumerate}
\item[{\rm (1)}]
Given a \ca{} $A$ that has $\calP$, every separable sub-\ca{} of $A$ is contained in a separable sub-\ca{} of $A$ that has $\calP$; and
\item[{\rm (2)}]
$\calP$ passes to (not necessarily sequential) inductive limits.
\end{enumerate}
\end{dfnIntro}

%==========================================================================================
\begin{thmIntro}[See \ref{prp:PureSepDet}, \ref{prp:CharPureness}]
Pureness of \ca{s} is separably determined.

In particular, a \ca{} is pure if and only if every separable sub-\ca{} is contained in a separable, pure sub-\ca{}.
\end{thmIntro}

%==========================================================================================
The combination of this separability reduction and an abstract analysis of comparison and divisibility for extensions in the setting of abstract Cuntz semigroups allows us to lift pureness across extensions.
While it is not clear that strict comparison and almost divisibility are individually preserved under extensions (\cref{qst:CompExt,qst:DivExt}), we show that their combined quantified versions, known as $(m,n)$-pureness (introduced in \cite{Win12NuclDimZstable}), behave well in this setting.
Specifically, we show that if an abstract Cuntz semigroup $S$ contains an ideal $I$ such that both~$I$ and~$S/I$ are $(m,n)$-pure for some $m,n$, then $S$ is $(m',n')$-pure for explicit values of $m'$ and $n'$ depending on $m$ and $n$ (\cref{prp:PureExtCu}).
In particular, an extension of pure \ca{s} is always $(1,1)$-pure, and hence pure by the reduction phenomenon established in \cite{AntPerThiVil24arX:PureCAlgs}.

\medskip

The paper is organized as follows. 
In \cref{sec:SepDet}, we introduce and develop the notion of separably determined properties, for both \ca{s} and Cuntz semigroups. 
In \cref{sec:CompDivSepDet}, we show that comparison and divisibility are separably determined and deduce that pureness is as well. 
In \cref{sec:CompDivExt}, we analyze how these properties behave in extensions of Cuntz semigroups and prove that pureness passes to extensions of \ca{s}. 
We conclude with applications to stable multiplier algebras and group \ca{s}.

%==========================================================================================
\subsection*{Acknowledgements}

%==========================================================================================
The second-named author thanks Ilijas Farah for helpful discussions on axiomatizable properties and the L\"{o}wenheim–Skolem condition, and Ping Ng for stimulating conversations regarding the structure of corona algebras.

%==========================================================================================
%==========================================================================================
\section{Separably determined properties of C*-algebras and Cu-semigroups}
\label{sec:SepDet}

%==========================================================================================
In this section we introduce the concept of being `separably determined' for properties of \ca{s} and \CuSgp{s}.
This notion is very useful for reducing arguments to the separable case, and this will be crucial in \cref{sec:CompDivExt}.
Many everyday properties of \ca{s} --- for example, stable rank one --- are separably determined. 
Moreover, if $\calP$ is a separably determined property of \CuSgp{s}, then the property of having a Cuntz semigroup with $\calP$ is separably determined for \ca{s};
see \cref{prp:RelateSepDet}.

%==========================================================================================
\begin{pgr}
Let $A$ be a \ca.
We denote by $\Sep(A)$ the collection of separable sub-\ca{s} of $A$, equipped with the partial order given by inclusion.
Every countable, directed subset $\calC \subseteq \Sep(A)$ admits a supremum, namely $\overline{\bigcup_{B \in \calC}B}$.
Restricting the model-theoretic definition to our case of interest, a family $\calC \subseteq \Sep(A)$ is called a \emph{club} if it is \emph{cofinal} (that is, for every $B_0\in\Sep(A)$ there exists $B\in\calC$ such that $B_0 \subseteq B$) and \emph{$\sigma$-complete} (that is, for every countable, directed subset $\calC' \subseteq \calC$ we have $\sup\calC' \in \calC$).

Let $\calP$ be a property of \ca{s} (for example, `has stable rank one').
One says that $\calP$ satisfies the \emph{L{\"o}wenheim-Skolem condition}, or that $\calP$ \emph{reflects to separable sub-\ca{s}} (\cite[Definition~7.3.1]{Far19BookSetThyCAlg}), if every \ca{} with $\calP$ contains a club of separable sub-\ca{s} that also have $\calP$.
\end{pgr}

%==========================================================================================
In the following definition, an inductive limit refers to the limit of an inductive system with not necessarily injective connecting maps and indexed by a directed set, which is not necessarily countable.

%==========================================================================================
\begin{dfn}
\label{dfn:SepDet}
We say that a property $\calP$ of \ca{s} is \emph{separably determined} if it satisfies the L{\"o}wenheim-Skolem condition and is preserved by inductive limits.
\end{dfn}

%==========================================================================================
\begin{rmk}
\label{rmk:SepDet}
Let $\calP$ be a property of \ca{s} that passes to inductive limits.
Given a \ca{} $A$, let $\Sep_{\calP}(A)$ denote the family of separable sub-\ca{s} that have $\calP$.
We note that $\Sep_{\calP}(A)$ is automatically $\sigma$-complete:
Indeed, if $\calC' \subseteq \Sep_{\calP}(A)$ is a countable, directed subset, then $\calC'$ is an inductive system indexed over itself whose inductive limit is $\sup\calC'$, and since $\calP$ passes to inductive limits, we get $\sup\calC' \in \Sep_{\calP}(A)$.

Thus, to ensure that~$\calP$ satisfies the L{\"o}wenheim-Skolem condition (and thus is separably determined), it suffices to assume that $\Sep_{\calP}(A)$ is cofinal whenever $A$ has $\calP$.
This shows that \cref{dfn:SepDetIntro} agrees with \cref{dfn:SepDet}.
\end{rmk}

%==========================================================================================
\begin{prp}
\label{prp:ClubWithProperty}
Let $\calP$ be a separably determined property of \ca{s}.
Then the following are equivalent for every \ca{} $A$:
\begin{enumerate}
\item[{\rm (1)}]
The \ca{} $A$ has $\calP$.
\item[{\rm (2)}]
The family of separable sub-\ca{s} of $A$ that have $\calP$ is a club.
\item[{\rm (3)}]
The \ca{} $A$ contains a club of separable sub-\ca{s} that have $\calP$. 
\end{enumerate}
\end{prp}
\begin{proof}
To show that~(1) implies~(2), assume that $A$ has $\calP$ and let $\Sep_{\calP}(A)$  be the family of separable sub-\ca{s} of $A$ that have $\calP$.
Since $\calP$ satisfies the L{\"o}wenheim-Skolem condition, $\calC$ contains a club and consequently is cofinal.
Further, $\Sep_{\calP}(A)$ is automatically $\sigma$-complete, as noted in \cref{rmk:SepDet}.

Clearly, (2) implies~(3).
To see that~(3) implies~(1), let $\calC \subseteq \Sep(A)$ be a club of separable sub-\ca{s} that have $\calP$.
Viewing $\calC$ as an inductive system indexed over itself, its limit is (isomorphic to) $A$, and since $\calP$ passes to inductive limits it follows that $A$ has $\calP$.
\end{proof}

%==========================================================================================
\begin{rmk}
The Downwards L{\"o}wenheim-Skolem theorem (\cite[Theorem~7.1.9, Lemma~7.3.3]{Far19BookSetThyCAlg}) shows that every axiomatizable property satisfies the L{\"o}wenheim-Skolem condition.
Further, every $\forall\exists$-axiomatizable property passes to inductive limits;
see \cite[Proposition~2.4.4(3)]{FarHarLupRobTikVigWin21ModelThy}.
It follows that every $\forall\exists$-axiomatizable property is separably determined, and \cite[Theorem 2.5.1]{FarHarLupRobTikVigWin21ModelThy} lists many such properties.

On the other hand, there are some natural properties of \ca{s} that are separably determined but that are not axiomatizable (or at least not known to be axiomatizable).
For example, the property of having real rank $\leq 1$ is separably determined (see the discussion after \cite[Definition~4.5]{Thi24RRExt}), but it remains open of this property is axiomatizable \cite[Question~3.9.3]{FarHarLupRobTikVigWin21ModelThy}.
\end{rmk}

%==========================================================================================
\begin{rmk}
Following Blackadar \cite[Definition~II.8.5]{Bla06OpAlgs}, one says that a property~$\calP$ of \ca{s} is \emph{separably inheritable} if for every \ca{} $A$ with $\calP$ the separable sub-\ca{s} with $\calP$ are cofinal in $\Sep(A)$ and $\calP$ passes to sequential inductive limits with injective connecting maps.
It thus follows from the definitions that every separably determined property is separably inheritable.
It is also known that every separably inheritable property satisfies the L{\"o}wenheim-Skolem condition.
The implications are shown in the following diagram:
\begin{center}
\begin{tikzcd}[column sep=10pt,row sep=12pt,arrows=Rightarrow]
\text{$\forall\exists$-axiomatizable} \ar[r] & \text{separably determined} \ar[d] \\
& \text{separably inheritable} \ar[r] & \text{L{\"o}wenheim-Skolem condition}
\end{tikzcd}
\end{center}
\end{rmk}

%==========================================================================================
\begin{pgr}
Let us recall the basic theory of Cuntz semigroups and their abstract counterparts.
A \emph{\CuSgp{}} is a positively ordered monoid such that
\begin{itemize}
\item[\axiomO{1}] 
increasing sequences have suprema;
\item[\axiomO{2}]
every element is the supremum of a $\ll$-increasing sequence;
\item[\axiomO{3}] 
addition is compatible with the relation $\ll$;
\item[\axiomO{4}] 
addition is compatible with taking suprema of increasing sequences;
\end{itemize}
where $\ll$ denotes the \emph{way-below relation} defined by setting $x\ll y$ if for every increasing sequence $(y_n)_n$ with $y \leq \sup_n y_n$ there exists $n\in\mathbb{N}$ such that $x\leq y_n$.

A \emph{generalized \CuMor} is an order-preserving monoid morphism that preserves suprema of increasing sequences.
If it also preserves the way-below relation then it is called a \emph{\CuMor}.

The category of \CuSgp{s} was introduced in \cite{CowEllIva08CuInv} to provide an abstract framework for the study of the \emph{Cuntz semigroup} of a \ca{} $A$. 
This invariant, denoted by $\Cu(A)$, is defined as the Cuntz equivalence classes of positive elements in the stabilization~$A\otimes\Cpct$.
Here, the Cuntz subsequivalence relation $\precsim$ and the Cuntz equivalence relation $\sim$ among positive elements $a$ and $b$ in a \ca{} are defined by setting $a\precsim b$ if there exists a sequence $(r_n)_n$ such that $a=\lim_n r_n br_n^*$, and by setting $a \sim b$ if $a\precsim b$ and $b \precsim a$.

When equipped with the partial order induced by $\precsim$ and the addition induced by addition of orthogonal positive elements in $A\otimes\Cpct$, the Cuntz semigroup becomes a \CuSgp{}; 
see \cite[Theorem 1]{CowEllIva08CuInv}.
Further, every $\ast$-homomorphism $A \to B$ between \ca{s} induces a natural \CuMor{} $\Cu(A)\to\Cu(B)$, and the Cuntz semigroup functor preserves inductive limits (\cite[Corollary 3.2.9]{AntPerThi18TensorProdCu}) and taking scales into account it also preserves (ultra)products (\cite[Theorem~7.5]{AntPerThi20CuntzUltraproducts}).~See also \cite{AraPerTom11Cu, GarPer23arX:ModernCu} for general expositions. 

A basis of a \CuSgp{} $S$ is a subset $B \subseteq S$ such that every element in $S$ can be written as the supremum of an increasing sequence of elements in $B$.
One says that $S$ is \emph{separable} or \emph{countably based} if it contains a countable basis.
The Cuntz semigroup of a separable \ca{} is separable;
see \cite[Lemma~1.3]{AntPerSan11PullbacksCu}.
\end{pgr}

%==========================================================================================
\begin{pgr}
Let $S$ be a \CuSgp.
As defined in \cite[Definition~4.1]{ThiVil21DimCu2}, a submonoid~$H$ of~$S$ is called a sub-\CuSgp{} if $H$ is a \CuSgp{} with the induced order and the inclusion map $H\to S$ is a \CuMor{}. 
We denote by $\Sep(S)$ the family of all separable (=countably based) sub-\CuSgp{s} of $S$.

Given a countable, directed subset $\calD \subseteq \Sep(S)$, the union $M:=\bigcup\calD=\bigcup_{T\in\calD} T$ is a submonoid of~$S$ such that every element in $M$ is the supremum (in $S$) of a $\ll$-increasing sequence of elements in $M$.
Therefore, the set
\[
\overline{M} := \left\{ \sup_n x_n \in S : (x_n)_n \text{ is an increasing sequence of elements in } M \right\}
\]
is a sub-\CuSgp{} by \cite[Corollary~4.12]{ThiVil21DimCu2}.
Using that $\calD$ is countable and that each $T \in \calD$ is separable, it follows that $\overline{M}$ is separable.
Then $\overline{M}$ is the supremum of $\calD$ in $\Sep(S)$.

As in the setting of \ca{s}, we define a family $\calD \subseteq \Sep(S)$ to be a \emph{club} if
\begin{enumerate}
\item[{\rm (i)}]  
$\calD$ is cofinal, that is, for every $T_0 \in \Sep(S)$ there exists $T \in \calD$ such that $T_0 \subseteq T$; and
\item[{\rm (ii)}] 
$\calD$ is $\sigma$-complete, that is, for every countable, directed subset $\calD'\subseteq\calD$ we have $\sup\calD'\in\calD$.
\end{enumerate}

Let $\calP$ be a property for \CuSgp{s} (for example, `has $0$-comparison').
We say that $\calP$ satisfies the \emph{L{\"o}wenheim-Skolem condition}, or that $\calP$ \emph{reflects to separable sub-\CuSgp{s}}, if every \CuSgp{} with $\calP$ contains a club of separable sub-\CuSgp{s} with $\calP$;
see \cite[Paragraph~5.2]{ThiVil21DimCu2}.
\end{pgr}

%==========================================================================================
Analogous to \cref{dfn:SepDet}, we define:

%==========================================================================================
\begin{dfn}
\label{dfn:SepDetCu}
We say that a property $\calP$ of \CuSgp{s} is \emph{separably determined} if it satisfies the L{\"o}wenheim-Skolem condition and passes to inductive limits.
\end{dfn}

%==========================================================================================
The following result follows a similar proof as in \cref{prp:ClubWithProperty}:

%==========================================================================================
\begin{prp}
\label{prp:ClubWithPropertyCu}
Let $\calP$ be a separably determined property of \CuSgp{s}. Then, for every \CuSgp{} $S$, we have that $S$ has $\calP$ if, and only if, the family of separable sub-\CuSgp{s} of $S$ that have $\calP$ is a club.
\end{prp}	

%==========================================================================================
The next result closely relates clubs of separable sub-\ca{s} of a \ca{} with clubs of separable sub-\CuSgp{s} of its Cuntz semigroup.

%==========================================================================================
\begin{lma}
\label{prp:RelateClubs}
Let $A$ be a \ca.
Then the collection $\calC_0$ of separable sub-\ca{s} $B \subseteq A$ such that the inclusion $B \to A$ induces an order-embedding $\Cu(B) \to \Cu(A)$ is a club in $\Sep(A)$.

Further, the family $\calD_0$ formed by the induced monoids $\Cu (B)$ ---when viewed as sub-\CuSgp{s} of $\Cu (A)$--- is a club in $\Sep(\Cu(A))$, and the induced map $\alpha\colon\calC_0\to\calD_0$ preserves the order and suprema of countable, directed sets.

If $\calD \subseteq \Sep(\Cu(A))$ is a club, then so is $\alpha^{-1}(\calD_0 \cap \calD) \subseteq \Sep(A)$.
\end{lma}
\begin{proof}
It was shown in \cite[Proposition~6.1]{ThiVil21DimCu2} that $\calC_0$ and $\calD_0$ are clubs and that $\alpha\colon\calC_0\to\calD_0$ preserves the order and suprema of countable, directed sets.
We will repeatedly use that the intersection of finitely many (even at most countably many) clubs is again a club;
see \cite[Proposition~6.2.9]{Far19BookSetThyCAlg}.

Let $\calD \subseteq \Sep(\Cu(A))$ be a club, and set $\calC := \alpha^{-1}(\calD_0 \cap \calD)$.
We need to show that $\calC$ is a club.
Using that $\alpha$ preserves suprema of countable, directed sets, we deduce that $\calC$ is $\sigma$-complete.

To show that~$\calC$ is cofinal, let $B_0\in \Sep(A)$.
Since $\calC_0$ is a club (and thus cofinal), we may assume that $B_0 \in \calC_0$.
Since $\calD\cap\calD_0$ is a club (and thus cofinal), there exists $H_1 \in \calD\cap\calD_0$ with $\Cu(B_0)\subseteq H_1$.
Pick $C_1 \in \calC$ with $\Cu(C_1)=H_1$.
Using that~$\calC_0$ is cofinal find $B_1 \in \calC_0$ containing $B_0$ and $C_1$.
Then $B_0 \subseteq B_1$ and $\Cu(B_0) \subseteq H_1 \subseteq \Cu(B_1)$.

Proceeding inductively, we find $H_1,H_2,\ldots \in \calD$ and $B_1,B_2,\ldots \in \calC_0$ such that
\[
\Cu(B_0) 
\subseteq H_1
\subseteq \Cu(B_1)
\subseteq H_2
\subseteq \Cu(B_2)
\subseteq \ldots, \andSep
B_0 \subseteq B_1 \subseteq B_2 \subseteq \ldots.
\]

Set $B := \overline{\bigcup_n B_n}$, which is the supremum of the increasing sequence $(B_n)_n$ in~$\Sep(A)$.
By construction, we have $B_0 \subseteq B$.
Since $\alpha$ preserves suprema of countable, directed sets (in particular, of increasing sequences), we have
\[
\alpha(B)=\sup_n \alpha(B_n) = \sup_n H_n.
\]
Since $\calD$ is a club, we have $\sup_n H_n \in \calD$, and thus $B \in \calC$, as desired.
\end{proof}

%==========================================================================================
\begin{prp}
\label{prp:RelateLS}
If $\calP$ is a property of \CuSgp{s} satisfying the L{\"o}wenheim-Skolem condition, then the property of having a Cuntz semigroup with $\calP$ satisfies the L{\"o}wenheim-Skolem condition.
\end{prp}
\begin{proof}
To show that the property of having a Cuntz semigroup with $\calP$ satisfies the L{\"o}wenheim-Skolem condition, let $A$ be a \ca{} such that $\Cu(A)$ has $\calP$.
Let the clubs $\calC_0\subseteq\Sep(A)$ and $\calD_0\subseteq\Sep(\Cu(A))$ and the map $\alpha\colon\calC_0\to\calD_0$ be as in \cref{prp:RelateClubs}.
Since $\calP$ satisfies the L{\"o}wenheim-Skolem condition, there exists a club $\calD\subseteq\Sep(\Cu(A))$ of separable sub-\CuSgp{s} with $\calP$.

By \cref{prp:RelateClubs}, $\alpha^{-1}(\calD_0 \cap \calD)$ is a club in $\Sep(A)$.
Note that $\alpha^{-1}(\calD_0 \cap \calD)$ consists of separable sub-\ca{s} of $A$ whose Cuntz semigroup has $\calP$, as desired.
\end{proof}

%==========================================================================================
\begin{thm}
\label{prp:RelateSepDet}
If $\calP$ is a separably determined property of \CuSgp{s}, then the property of having a Cuntz semigroup with $\calP$ is separably determined.
\end{thm}
\begin{proof}
By \cref{prp:RelateLS}, the property $\calQ$ of having a Cuntz semigroup with $\calP$ satisfies the L{\"o}wenheim-Skolem condition.
By \cite[Corollary~3.2.9]{AntPerThi18TensorProdCu}, the Cuntz semigroup functor preserves inductive limits (over not necessarily countable index sets). 
It follows that $\calQ$ passes to inductive limits.
\end{proof}

%==========================================================================================
%==========================================================================================
\section{Comparison and divisibility are separably determined}
\label{sec:CompDivSepDet}

%==========================================================================================
In this section we show that pureness is separably determined for \ca{s}
see \cref{prp:PureSepDet}.
More specifically, for each $m\in\NN$ we show that $m$-comparison is separably determined (\cref{prp:CompSepDet}), and for each $n\in\NN$ we show that $n$-almost divisibility is separably determined (\cref{prp:DivSepDet}).

\medskip

We first recall the definition of pureness, introduced by Winter for simple \ca{s} in \cite[Definitions~3.1~and~3.5]{Win12NuclDimZstable} and later adapted to the non-simple setting by Robert-Tikuisis \cite[Paragraph~2.3]{RobTik17NucDimNonSimple};
see also \cite{AntPerThiVil24arX:PureCAlgs}.
We will use the relation $<_s$ on a \CuSgp, which is defined by setting $x<_s y$ if there exists $n\in\NN$ such that $(n+1)x\leq ny$.

%==========================================================================================
\begin{dfn}
\label{dfn:pure}
Let $S$ be a \CuSgp, and let $m,n \in \NN$.
One says that $S$ has \emph{$m$-comparison} if for any $x,y_0,\ldots ,y_m \in S$ such that $x<_s y_j$ for each $j=0,\ldots ,m$, then $x\leq y_0+\ldots +y_m$.
Further, one says that $S$ is \emph{$n$-almost divisible} if for any $k\in\NN$ and any $x',x \in S$ such that $x'\ll x$ there exists an element $y$ such that $ky\leq x$ and $x'\leq (k+1)(n+1)y$.

We say that $S$ is \emph{$(m,n)$-pure} if it has $m$-comparison and is $n$-almost divisible. Further, $S$ is said to be \emph{pure} if it is $(0,0)$-pure.

A \ca{} $A$ is termed \emph{$(m,n)$-pure} if its Cuntz semigroup $\Cu(A)$ is $(m,n)$-pure; and $A$ is \emph{pure} if $\Cu(A)$ is pure (See \cref{rmk:MnPureImpPure} below.)
\end{dfn}

%==========================================================================================
\begin{rmk}
\label{rmk:MnPureImpPure}
The property of $0$-comparison is often referred to as \emph{almost unperforation}, and the Cuntz semigroup of a \ca{} $A$ is almost unperforated if and only if $A$ has \emph{strict comparison of positive elements by quasitraces}; 
see \cite[Proposition~6.2]{EllRobSan11Cone} and the discussions in \cite[Paragraph~3.4]{AntPerThiVil24arX:PureCAlgs}.
The property of $0$-almost divisibility is usually referred to as \emph{almost divisibility}.
Thus, a \ca{} is pure if (by definition) its Cuntz semigroup is almost unperforated and almost divisible.

It is easy to see that $(m,n)$-pureness implies $(m',n')$-pureness for all $m' \geq m$ and $n' \geq n$, both for \ca{s} and at the level of \CuSgp{s}.
In the converse direction, \ca{s} exhibit a reduction phenomenon:
a \ca{} is automatically pure whenever it is $(m,n)$-pure for some $m,n \in \NN$. 
This was established in \cite[Theorem~5.7]{AntPerThiVil24arX:PureCAlgs}, building on earlier results in the simple case \cite[Theorem~10.5]{AntPerRobThi24TracesUltra}. 
For simple, unital \ca{s} with locally finite nuclear dimension, the phenomenon also follows from Winter’s $\calZ$-stability theorem \cite{Win12NuclDimZstable} in combination with R{\o}rdam's theorem that $\calZ$-stable \ca{s} are pure \cite{Ror04StableRealRankZ}.
\end{rmk}

%==========================================================================================
Let $S$ be a \CuSgp{}, and let $\varphi_\lambda\colon S_\lambda \to S$ be \CuMor{s} from some \CuSgp{s} $S_\lambda$, for $\lambda\in\Lambda$.
Following \cite[Definition~3.1]{ThiVil21DimCu2}, we say that $S$ is \emph{approximated} by the family $(S_\lambda,\varphi_\lambda)_{\lambda\in\Lambda}$ if the following holds:
For any finite index sets $J$ and $K$, any $x_j',x_j \in S$ for $j \in J$ (thought of as variables), and any functions $m_k,n_k \colon J \to \NN$ for $k \in K$ (thought of as relations) such that
\[
x_j' \ll x_j \text{ for $j \in J$}, \andSep
\sum_{j \in J} m_k(j)x_j \ll \sum_{j \in J} n_k(j)x_j'  \text{ for $k \in K$}
\]
there exists an index $\lambda\in\Lambda$ and $y_j \in S_\lambda$ for $j \in J$ such that
\[
x_j' \ll \varphi_\lambda(y_j) \ll x_j \text{ for $j \in J$}, \andSep
\sum_{j \in J} m_k(j)y_j \ll \sum_{j \in J} n_k(j)y_j \text{ for $k \in K$}.
\]

%==========================================================================================
\begin{lma}
\label{prp:CompApprox}
For each $m\in\NN$, the property of $m$-comparison passes to approximated \CuSgp{s}.
\end{lma}
\begin{proof}
Let $m\in\NN$, and let $S$ be a \CuSgp{} that is approximated by the family $(S_\lambda,\varphi_\lambda)_{\lambda\in\Lambda}$.
Assume that each $S_\lambda$ has $m$-comparison.
To show that $S$ has $m$-comparison, let $x,y_0,\ldots ,y_m \in S$ satisfy $x<_s y_j$ for each $j=0,\ldots,m$.
We need to show that $x\leq y_0+\ldots +y_m$.

Let $x' \in S$ satisfy $x' \ll x$, and pick $z \in S$ such that $x' \ll z \ll x$.
Set $z' := x'$.
For each $j$, first pick $n_j\in\NN$ such that $(n_j+1)x \leq n_j y_j$, and then choose $y_j' \in S$ such that
\[
y_j' \ll y_j, \andSep
(n_j+1)z \ll n_j y_j'.
\]
Since $(S_\lambda,\varphi_\lambda)_{\lambda\in\Lambda}$ approximates $S$, there exists $\lambda$ and $u,w_0,\ldots,w_m \in S_\lambda$ such that
\[
z' \ll \varphi_\lambda(u) \ll z, \andSep
y_j' \ll \varphi_\lambda(w_j) \ll y_j \text{ for $j=0,\ldots,m$}
\]
and
\[
(n_j+1)u \ll n_j w_j \text{ for $j=0,\ldots,m$}.
\]
Then $u <_s w_0,\ldots,w_m$.
Since $S_\lambda$ has $m$-comparison, we get $u \leq w_0+\ldots+w_m$, and then
\[
x' = z' 
\leq \varphi_\lambda(u) 
\leq \sum_{j=0}^m \varphi_\lambda(w_j) 
\leq \sum_{j=0}^m y_j.
\]
Since this holds for every $x'$ way-below $x$, we get $x \leq y_0+\ldots+y_m$, as desired.
\end{proof}

%==========================================================================================
\begin{prp}
\label{prp:CompSepDet}
For each $m\in\NN$, the property of $m$-comparison is separably determined for \CuSgp{s}.
\end{prp}
\begin{proof}
Let $m\in\NN$.
It follows directly from the definitions that $m$-comparison passes to sub-\CuSgp{s}.
Therefore, $m$-comparison satisfies the L{\"o}wenheim-Skolem condition.
Further, by \cref{prp:CompApprox}, $m$-comparison passes to approximated \CuSgp{s}, and therefore passes to inductive limits by \cite[Proposition~3.5]{ThiVil21DimCu2}.
\end{proof}

%==========================================================================================
\begin{lma}
\label{prp:DivApprox}
For each $n\in\NN$, the property of $n$-almost divisibility passes to approximated \CuSgp{s}.
\end{lma}
\begin{proof}
Let $n\in\NN$, and let $S$ be a \CuSgp{} that is approximated by the family $(S_\lambda,\varphi_\lambda)_{\lambda\in\Lambda}$.
Assume that each $S_\lambda$ is $n$-almost divisible.
To show that $S$ is $n$-almost divisible, let $x',x \in S$ satisfy $x' \ll x$.
We need to find $y \in S$ such that $ky \leq x$ and $x' \leq (k+1)(n+1)y$.

We directly obtain $\lambda$ and $u \in S_\lambda$ such that
\[
x' \ll \varphi_\lambda(u) \ll x.
\]
Choose $u' \in S_\lambda$ such that
\[
u' \ll u, \andSep
x' \ll \varphi_\lambda(u').
\]
Using that $S_\lambda$ is $n$-almost divisible, we obtain $w \in S_\lambda$ such that $kw \leq u$ and $u' \leq (k+1)(n+1)w$.
Then $y:=\varphi_\lambda(w)$ has the desired properties.
\end{proof}

%==========================================================================================
\begin{prp}
\label{prp:DivSepDet}
For each $n\in\NN$, the property of $n$-almost divisibility is separably determined for \CuSgp{s}.
\end{prp}
\begin{proof}
Let $n\in\NN$.
By \cref{prp:DivApprox}, $n$-almost divisibility passes to approximated \CuSgp{s}, and therefore passes to inductive limits by \cite[Proposition~3.5]{ThiVil21DimCu2}.
It remains to show that it satisfies the L{\"o}wenheim-Skolem condition.

Let $S$ be an $n$-almost divisible \CuSgp{}, and let $\calC$ denote the family of separable sub-\CuSgp{s} that are $n$-almost divisible.
Using that $n$-almost divisibility passes to inductive limits, it follows that $\calC$ is $\sigma$-complete.
To show that~$\calC$ is cofinal, let $T_0 \in \Sep(S)$.
We need to find $T \in \Sep(S)$ with $T_0 \subseteq T$ and such that~$T$ is $n$-almost divisible.
Choose a countable basis $B_0 \subseteq T_0$.
The proof follows a standard technique, as for example the one used in \cite[Proposition~5.3]{ThiVil21DimCu2}.

We inductively choose an increasing sequence $(T_j)_j$ in $\Sep(S)$ and a countable basis $B_j \subseteq T_j$ such that for each $j \in \NN$ the following holds:
\begin{equation}
\tag{$\ast$}
\parbox{11cm}{\emph{For every $x',x \in B_j$ satisfying $x' \ll x$ and every $k\in\NN$, there exists $y_k \in B_{j+1}$ such that $ky_k \leq x$ and $x' \leq (k+1)(n+1)y_k$.}}
\end{equation}

We already chose $T_0$ and $B_0$.
Let $j\in\NN$ and assume we chose $T_j$ and $B_j$.
Set
\[
I_j := \big\{ (x',x) \in B_j \times B_j : x' \ll x \big\},
\]
which is countable since $B_j$ is. 

For each $i=(x',x) \in I_j$ and $k\in\NN$, since $S$ is $n$-almost divisible, we can pick $y_{i,k} \in S$ with $ky_{i,k} \leq x$ and $x' \leq (k+1)(n+1)y_{i,k}$. 
Now $B_j \cup \{ y_{i,k} : i \in I_j, k\in\NN \}$ is a countable subset of $S$, which by \cite[Lemma~5.1]{ThiVil21DimCu2} is contained in some separable sub-\CuSgp{} $T_{j+1} \subseteq S$.
Since $B_j$ is a basis for $T_j$, we have $T_j \subseteq T_{j+1}$.
Pick a countable basis $B_{j+1} \subseteq T_{j+1}$ that contains $B_j$.

Set $T:=\sup_j T_j \in \Sep(S)$.
Clearly $T_0 \subseteq T$, and we claim that $T$ is $n$-almost divisible.
To verify this, let $x' \ll x$ in $T$, and let $k\in\NN$.
Since every element in $T$ is the supremum of an increasing sequence in $\bigcup_j T_j$, we see $\bigcup_j B_j$ is a basis for $T$, which allows us to find $j$ and $u',u \in B_j$ such that $x' \leq u' \ll u \leq x$.
Applying $(\ast)$, we obtain the desired element in $T_{j+1} \subseteq T$.
\end{proof}

%==========================================================================================
\begin{prp}
\label{prp:PureSepDetCu}
For each $m,n\in\NN$, the property of $(m,n)$-pureness is separably determined for \CuSgp{s}.
\end{prp}
\begin{proof}
In general, the conjunction of finitely (or even countably) many separably determined properties is itself separably determined.
Therefore, the result follows by combining \cref{prp:CompSepDet,prp:DivSepDet}.
\end{proof}

%==========================================================================================
\begin{thm}
\label{prp:PureSepDet}
Pureness of \ca{s} is separably determined.
\end{thm}
\begin{proof}
This follows from \cref{prp:RelateSepDet} with \cref{prp:PureSepDetCu} for $m=n=0$.
\end{proof}

%==========================================================================================
\begin{rmk}
The result that a \ca{} is pure whenever it is approximated by pure sub-\ca{s} can also be deduced from \cite[Corollary~4.8]{BosVil25PureHom}.
\end{rmk}

%==========================================================================================
\begin{cor}
\label{prp:CharPureness}
A \ca{} is pure if and only if every separable sub-\ca{} is contained in a separable, pure sub-\ca{}.

A \ca{} has strict comparison (of positive elements by quasitraces) if and only if every separable sub-\ca{} is contained in a separable sub-\ca{} with strict comparison.
\end{cor}

%==========================================================================================
\begin{rmk}
In \cite[Sections~3,4]{AntPerThiVil24arX:PureCAlgs} the study of the unquantized notion of \emph{controlled comparison} (resp. \emph{functional divisibility}) was crucial to obtain the dimension reduction phenomenon \cite[Theorem~C]{AntPerThiVil24arX:PureCAlgs}. Controlled comparison (resp. functional divisibility) is weaker than $m$-comparison for any $m$ (resp. $n$-almost divisibility for any $n$); see \cite[Proposition~4.9]{AntPerThiVil24arX:PureCAlgs} and \cite[Proposition~10.3]{AntPerRobThi24TracesUltra}.

We have shown in this section that the notions of $m$-comparison and $n$-almost divisibility are separably determined. However, this is not true for controlled comparison. 
Indeed, some Villadsen algebras of first type (e.g. the one in \cite{Tom08ClassificationNuclear}) fail to have controlled comparison \cite[Example~5.16]{AntPerThiVil24arX:PureCAlgs}, although they are built as inductive limits whose blocks have controlled comparison (since each block has $m$-comparison for some $m$). 

In contrast, functional divisibility passes to limits. In fact, one can use techniques similar to those in this section to show that it is separably determined.
\end{rmk}

%==========================================================================================
%==========================================================================================
\section{Comparison and divisibility of extensions} 
\label{sec:CompDivExt}

%==========================================================================================
In this section we prove the main result of the paper (\cref{prp:PureExt}): pureness passes to extensions of \ca{s}.
At the level of \CuSgp{s}, we show that the property of being $(m,n)$-pure for some $m,n\in\NN$ passes to extensions;
see \cref{prp:PureExtCu}.
However, the `degree of pureness' may potentially decrease.
In particular, we establish that an extension of pure (that is, $(0,0$)-pure) \CuSgp{s} is $(1,1)$-pure.
In the \ca{} setting, this is sufficient because, as already mentioned, every $(1,1)$-pure \ca{} is pure (\cite[Theorem~5.7]{AntPerThiVil24arX:PureCAlgs}).

\medskip

%==========================================================================================
In the results below, we use additional order-theoretic properties for \CuSgp{s} that are labelled \axiomO{5}-\axiomO{8}. Except for \axiomO{5}, the other properties are needed as assumptions in other results in the literature that will be quoted. Therefore, we will only recall the formulation of \axiomO{5} when needed (namely, in the proof of \cref{prp:DivExtSep}). Let us remark that the Cuntz semigroup of every \ca{} satisfies \axiomO{5}-\axiomO{8}: specifically, we refer to \cite[Proposition~4.6]{AntPerThi18TensorProdCu} for \axiomO{5}, to \cite[Proposition~5.1.1]{Rob13Cone} for \axiomO{6}, to \cite[Proposition~2.2]{AntPerRobThi21Edwards} for \axiomO{7}, and to \cite[Theorem~7.4]{ThiVil24NowhereScattered} for \axiomO{8}.

Given a \CuSgp{} $S$, an \emph{ideal} is a submonoid $I \subseteq S$ that is closed under suprema of increasing sequences and that is downward hereditary (that is, if $x,y \in S$ satisfy $x \leq y$ and $y \in I$, then $x \in I$).
Given an ideal $I$, the relations $\leq_I$ and $\sim_I$ on~$S$ are defined by setting $x \leq_I y$ if there exists $w \in I$ such that $x \leq y+w$, and by setting $x \sim_I y$ if $x \leq_I y$ and $y \leq_I x$.
Then $\sim_I$ is an equivalence relation, and the set of equivalence classes is denoted by $S/I$.
The addition on $S$ induces an addition on $S/I$, and the relation $\leq_I$ induces a partial order on $S/I$.
This gives $S/I$ the structure of a \CuSgp{}; see \cite[Section~5.1]{AntPerThi18TensorProdCu} for details. Note that, upon replacing $w$ by $\infty w:=\sup_n nw$ in the definition of $\leq_I$, we may assume that $w\in I$ is an idempotent, i.e. $2w=w$.

%==========================================================================================
\begin{lma}
\label{prp:CompExtSep}
Let $S$ be a separable \CuSgp{} satisfying \axiomO{5}-\axiomO{7}, and let~$I$ be an ideal in~$S$.
Assume that $S/I$ has $m_1$-comparison, and that $I$ has $m_2$-comparison.
Then $S$ has $(m_1+m_2+1)$-comparison.
\end{lma}
\begin{proof}
To verify that the \CuSgp{} $S$ has $(m_1+m_2+1)$-comparison, let $x,y_0,\ldots,y_{m_1},z_0,\ldots,z_{m_2} \in S$ be such that $x<_s y_j$ for $j=0,\ldots,m_1$ and $x <_s z_j$ for $j=0,\ldots,m_2$.
We need to show that $x \leq y_0+\ldots+y_{m_1}+z_0+\ldots+z_{m_2}$.

The relations $x<_s y_j$ still hold when passing to the quotient $S/I$. 
Using that~$S/I$ has $m_1$-comparison, we get that
\begin{equation}
\label{eq:CompExt-1}
x \leq y_0+\ldots+y_{m_1}+w
\end{equation}
for some idempotent $w \in I$. 
%By replacing~$w$ with $\infty w$ (which is defined as $\infty w := \sup_n nw$), we may assume that~$w$ is idempotent, that is, that $w=2w$. 

As shown in \cite[Theorem~2.4]{AntPerRobThi21Edwards}, in separable \CuSgp{s} satisfying \axiomO{7}, the infimum between an arbitrary element and an idempotent element exists.
In our setting, it follows that $x \wedge w$ exists in $S$.
Moreover, by applying \cite[Theorem~2.5~(ii)]{AntPerRobThi21Edwards} to \eqref{eq:CompExt-1}, it follows that
\begin{equation}
\label{eq:CompExt-2}
x \leq y_0+\ldots+y_{m_1}+(x\wedge w).
\end{equation}

By \cite[Theorem~2.5~(i)]{AntPerRobThi21Edwards}, the map $S \to I$, $z \mapsto (z\wedge w)$, is a generalized \CuMor.
It follows that $(x\wedge w)<_s(z_j\wedge w)$ for $j=0,\ldots,m_2$.
Using that $I$ has $m_2$-comparison, we get that
\[
(x\wedge w) 
\leq (z_0\wedge w)+\ldots+(z_{m_2}\wedge w)
\leq z_0+\ldots+z_{m_2}.
\]
Combining this estimate with \eqref{eq:CompExt-2}, we get 
\[
x 
\leq y_0+\ldots+y_{m_1}+(x\wedge w)
\leq y_0+\ldots+y_{m_1}+z_0+\ldots+z_{m_2},
\]
as desired.
\end{proof}

%==========================================================================================
\begin{prp}
\label{prp:OSepDet}
Each of the properties \axiomO{5}-\axiomO{8} is separably determined.
\end{prp}
\begin{proof}
Each of the properties \axiomO{5}-\axiomO{7} satisfies the L{\"o}wenheim-Skolem condition by \cite[Proposition~5.3]{ThiVil21DimCu2}, and passes to inductive limits by \cite[Corollary~3.6]{ThiVil21DimCu2}.
With similar methods, one shows that also \axiomO{8} is separably determined.
\end{proof}

%==========================================================================================
Let $I$ be an ideal in a \CuSgp{} $S$, and let $\pi_I \colon S \to S/I$ be the quotient map.
If $T \subseteq S$ is a sub-\CuSgp{}, then it is not clear that $\pi_I(T)$ is a sub-\CuSgp{} of $S/I$.
However, the next result shows that there are many separable sub-\CuSgp{s} of $S$ with this property.
The result is proved with similar methods as \cref{prp:RelateClubs} and \cite[Proposition~6.1]{ThiVil21DimCu2}.
We omit the details.

%==========================================================================================
\begin{lma}
\label{prp:RelateClubsQuot}
Let $S$ be a \CuSgp{}, let~$I$ be an ideal in~$S$, and let $\pi_I \colon S \to S/I$ be the quotient map.
Then the collection $\calC_0$ of separable sub-\CuSgp{s} $T \subseteq S$ such that $\pi_I(T)$ is a sub-\CuSgp{} of $S/I$ is a club in $\Sep(S)$.

Further, if $\calD\subseteq\Sep(S/I)$ is a club, then so is
\[
\calC := \big\{ T \in \calC_0 : \pi_I(T) \in \calD \big\}.
\]
\end{lma}

%==========================================================================================
\begin{lma}
\label{prp:RelateClubsIdl}
Let~$I$ be an ideal in a \CuSgp{}~$S$, and let $\calD \subseteq \Sep(I)$ be a club.
Then $\{T \in \Sep(S) : I \cap T \in \calD\}$ is a club in $\Sep(S)$.
\end{lma}
\begin{proof}
This is the \CuSgp{} analog of \cite[Lemma~3.2(1)]{Thi23grSubhom} and can be proved with similar methods.
We omit the details.
\end{proof}

%==========================================================================================
\begin{prp}
\label{prp:CompExt}
Let $S$ be a \CuSgp{} satisfying \axiomO{5}-\axiomO{7}, and let~$I$ be an ideal in~$S$.
Assume that $S/I$ has $m_1$-comparison, and that $I$ has $m_2$-comparison.
Then $S$ has $(m_1+m_2+1)$-comparison.
\end{prp}
\begin{proof}
Let $\pi_I \colon S \to S/I$ denote the quotient map.
Let $\calC$ be the family of separable sub-\CuSgp{s} of $S/I$ that have $m_1$-comparison.
Since $m_1$-comparison is separably determined (\cref{prp:CompSepDet}), it follows from \cref{prp:ClubWithPropertyCu} that $\calC$ is a club.
Similarly, the family $\calD$ of separable sub-\CuSgp{s} of $I$ that have $m_2$-comparison is a club.
By \cref{prp:RelateClubsQuot,prp:RelateClubsIdl}, the collections
\begin{align*}
\calC' 
&:= \big\{T \in \Sep(S) : \pi_I(T) \text{ is a sub-\CuSgp{}}, \pi_I(T)  \in \calC \big\}, \andSep \\
\calD' 
&:= \big\{T \in \Sep(S) : T \cap I \in \calD \big\}
\end{align*}
are clubs in $\Sep(S)$.

By \cref{prp:OSepDet}, the families $\calE_{\axiomO{5}}$, $\calE_{\axiomO{6}}$ and $\calE_{\axiomO{7}}$ of separable sub-\CuSgp{s} of $S$ that satisfy \axiomO{5}, \axiomO{6} and \axiomO{7}, respectively, are clubs.
Set
\[
\calF := \calC' \cap \calD' \cap \calE_{\axiomO{5}} \cap \calE_{\axiomO{6}} \cap \calE_{\axiomO{7}}.
\]
Since the intersection of at most countably many clubs is again a club, it follows that $\calF$ is a club in $\Sep(S)$.

Let $T \in \calF$.
Then $T$ is a separable \CuSgp{} satisfying \axiomO{5}-\axiomO{7}, and $T \cap I$ is an ideal in $T$ such that $T/(I \cap T)$ has $m_1$-comparison (since it is isomorphic to $\pi_I(T)$ and $T \in \calC'$) and such that $I \cap T$ has $m_2$-comparsion (since $T \in \calD'$).
By \cref{prp:CompExtSep}, $T$ has $(m_1+m_2+1)$-comparison.
Since $\calF$ is a club in $S$, and since $(m_1+m_2+1)$-comparison is separably determined, it follows that $S$ has $(m_1+m_2+1)$-comparison.
\end{proof}

%==========================================================================================
We now turn to divisibility properties of extensions.

%==========================================================================================
\begin{lma}
\label{prp:LiftDiv}
Let $S$ be a separable \CuSgp{} satisfying \axiomO{5}-\axiomO{8}, and let~$I$ be an ideal in~$S$, and assume that $S/I$ is $n$-almost divisible.
Let $x' \ll x$ in $S$, and let $k\in\NN$.
Then there exist $u \in S$ and an idempotent $w \in I$ such that
\[
ku \ll x, \andSep
x' \ll (k+1)(n+1)u + w.
\]
\end{lma}
\begin{proof}
Find $x'' \in S$ such that $x' \ll x'' \ll x$.
Since the quotient map $S \to S/I$ is a \CuMor{}, the image of $x''$ is way-below the image of $x$ in $S/I$.
Using that $S/I$ is $n$-almost divisible, we find $y \in S$ such that
\[
ky \leq_I x, \andSep
x'' \leq_I (k+1)(n+1)y.
\]
This gives idempotent elements $w_1,w_2 \in I$ such that
\[
ky \leq x + w_1, \andSep
x'' \leq (k+1)(n+1)y + w_2.
\]
%Set $w:=\infty(w_1+w_2)$, which is an idempotent in $I$.
Set $w:=w_1+w_2$. Using that $x' \ll x''$, we then have
\[
ky \leq x + w, \andSep
x' \ll (k+1)(n+1)y + w.
\]
Choose $y'',y' \in S$ such that
\[
y'' \ll y' \ll y, \andSep 
x' \ll (k+1)(n+1)y'' + w.
\]
Then $ky' \ll x + w$ and $y'' \ll y'$, which allows us to apply \cite[Proposition~7.8]{ThiVil24NowhereScattered} to obtain $u \in S$ such that
\[
ku \ll x, \quad
y'' \ll u + w, \andSep
u \ll y' + w.
\]
Using that $w$ is idempotent, we get
\[
x' \ll (k+1)(n+1)y'' + w
\leq (k+1)(n+1)u + w,
\]
as desired.
\end{proof}

%==========================================================================================
\begin{lma}
\label{prp:DivExtSep}
Let $S$ be a separable \CuSgp{} satisfying \axiomO{5}-\axiomO{8}, and let~$I$ be an ideal in~$S$.
Assume that $S/I$ is $n_1$-almost divisible, and that $I$ is $n_2$-almost divisible.
Then $S$ is $n$-almost divisible for $n=\max\{2n_1+1,2n_2+1\}$.
\end{lma}
\begin{proof}
Let $x'',x \in S$ with $x'' \ll x$, and let $k \in \NN$. 
We need to find $d \in S$ such that $kd \leq x$ and $x'' \leq (k+1)(n+1)d$.
Pick $x' \in S$ such that $x'' \ll x' \ll x$.
Applying \cref{prp:LiftDiv} for $2k$ and $x' \ll x$, we obtain $z \in S$ and an idempotent $w \in I$ such that
\begin{equation}
\label{eq:DivExt-1}
2kz \ll x, \andSep
x' \ll (2k+1)(n_1+1)z + w.
\end{equation}
Find $z'',z' \in S$ such that
\[
z'' \ll z' \ll z, \andSep
x' \ll (2k+1)(n_1+1)z'' + w.
\]

The `almost algebraic order' property \axiomO{5} means that for $r',r,s',s,t \in S$ with $r+s \leq t$ and $r' \ll r$, $s' \ll s$, there exists $u \in S$ such that $r' + u \leq t \leq r+u$ and $s' \ll u$.
In our setting, we have
\[
kz' + kz \leq x, \quad
kz'' \ll kz', \andSep
kz' \ll kz.
\]
Applying \axiomO{5}, we obtain $u \in S$ such that
\[
kz'' + u \leq x \leq kz' + u, \andSep
kz' \ll u.
\]
Then
\[
x \leq 2u.
\]

As in the proof of \cref{prp:CompExtSep}, we use that infima between $w$ and arbitrary elements in $S$ exist.
By applying \cite[Theorem~2.5~(ii)]{AntPerRobThi21Edwards} to \eqref{eq:DivExt-1} at the first step, and by using that $s \mapsto (s\wedge w)$, is a generalized \CuMor{} (\cite[Theorem~2.5~(i)]{AntPerRobThi21Edwards}) at the second step, we get
\[
x' 
\leq (2k+1)(n+1)z'' + (x' \wedge w)
\leq (2k+1)(n+1)z'' + 2(u \wedge w).
\]
Choose $u' \in S$ such that
\[
u' \ll (u \wedge w), \andSep
x'' \leq (2k+1)(n+1)z'' + 2u'.
\]

Using that $I$ is $n_2$-almost divisible for $k$ and $u' \ll (u \wedge w)$, we obtain $v \in I$ such that
\[
kv \leq u \wedge w, \andSep
u' \leq (k+1)(n_2+1)v.
\]
Set $d:=z''+v$.
Then
\[
kd 
= kz'' + kv
\leq kz'' + (u \wedge w)
\leq kz'' + u
\leq x.
\]
Further,
\begin{align*}
x''
&\leq (2k+1)(n_1+1)z'' + 2u' \\
&\leq 2(k+1)(n_1+1)z'' + 2(k+1)(n_2+1)v \\
&\leq (k+1)(n+1)z''+(k+1)(n+1)v \\
&= (k+1)(n+1)d.
\end{align*}
This shows that $d$ has the desired properties.
\end{proof}

%==========================================================================================
\begin{prp}
\label{prp:DivExt}
Let $S$ be a \CuSgp{} satisfying \axiomO{5}-\axiomO{8}, and let~$I$ be an ideal in~$S$.
Assume that $S/I$ is $n_1$-almost divisible, and that $I$ is $n_2$-almost divisible.
Then $S$ is $n$-almost divisible for $n=\max\{2n_1+1,2n_2+1\}$.
\end{prp}
\begin{proof}
This is proved analogous to \cref{prp:CompExt}.
We omit the details.
\end{proof}

%==========================================================================================
\begin{lma}
\label{prp:PureIdlQuot}
Let $I$ be an ideal in a \CuSgp{} $S$.
If $S$ has $m$-comparison, then so do $I$ and $S/I$.
If $S$ is $n$-almost divisible, then so are $I$ and $S/I$.
%Then $I$ and $S/I$ are $(m,n)$-pure.
\end{lma}
\begin{proof}
It follows directly from the definitions that $m$-comparison and $n$-almost divisibility pass to ideals.
Assuming that $S$ has $m$-comparison, let us verify that~$S/I$ has $m$-comparison as well.
Let $x,y_0,\ldots,y_m \in S$ be elements such that the image of $x$ in $S/I$ is $<_s$-below the image of $y_j$ in $S/I$ for each $j$.
For each $j$, we obtain $n_j \in \NN$ and $w_j \in I$ such that
\[
(n_j+1)x \leq n_j y_j + w_j.
\]
Then $(n_j+1)x \leq n_j(y_j+w_j)$ and so $x<_s(y_j+w_j)$.
Using that $S$ has $m$-comparison, it follows that 
\[
x \leq (y_0+w_0)+\ldots+(y_m+w_m) = (y_0+\ldots+y_m)+(w_0+\ldots+w_m).
\]
Since $w_0+\ldots+w_m$ belongs to $I$, we get $x \leq_I y_0+\ldots+y_m$, as desired.

Assuming that $S$ is $n$-almost divisible, let us verify that~$S/I$ is $n$-almost divisible.
Let $\pi_I \colon S \to S/I$ denote the quotient map, and let $u',u \in S/I$ with $u' \ll u$, and let $k\in\NN$.
Pick a lift $x \in S$ of $u$.
Using that $\pi_I$ preserves suprema of increasing sequences, we find $x' \in S$ with $x' \ll x$ and $u' \leq \pi_I(x')$.
Using that $S$ is $n$-almost divisible, we obtain $y \in S$ such that $ky \leq x$ and $x' \leq (k+1)(n+1)y$.
Then $\pi_I(y)$ has the desired properties.
\end{proof}

%==========================================================================================
\begin{thm}
\label{prp:PureExtCu}
Let $S$ be a \CuSgp{} satisfying \axiomO{5}-\axiomO{8}, and let~$I$ be an ideal in~$S$.
Then $S$ is $(m,n)$-pure for some $m,n\in\NN$ if and only if $I$ is $(m',n')$-pure for some $m',n'\in\NN$ and $S/I$ is $(m'',n'')$-pure for some $m'',n''\in\NN$.
%Assume that $S/I$ is $(m_1,n_1)$-pure and that $I$ is $(m_2,n_2)$-pure.
%Then $S$ is $(m,n)$-pure for $m=m_1+m_2+1$ and $n=\max\{2n_1+1,2n_2+1\}$.
\end{thm}
\begin{proof}
The forward implication follows from \cref{prp:PureIdlQuot}, and the backward implication follows from \cref{prp:CompExt,prp:DivExt}.
\end{proof}

%==========================================================================================
\begin{thm}
\label{prp:PureExt}
Let $I$ be a closed ideal in a \ca{} $A$.
Then $A$ is pure if and only if~$I$ and~$A/I$ are pure.
\end{thm}
\begin{proof}
By \cite{CiuRobSan10CuIdealsQuot}, we can identify $\Cu(I)$ with an ideal in $\Cu(A)$ such that $\Cu(A/I)$ is naturally isomorphic to $\Cu(A)/\Cu(I)$;
see also \cite[Section~5.1]{AntPerThi18TensorProdCu}.
Therefore, the forward implication follows from \cref{prp:PureIdlQuot}.

For the backward implications, assume that $I$ and $A/I$ are pure. Since, as mentioned before, Cuntz semigroups of \ca{s} always satisfy \axiomO{5}-\axiomO{8}, we may apply \cref{prp:CompExt,prp:DivExt} to deduce that $\Cu(A)$ is $(1,1)$-pure. By \cite[Theorem~5.7]{AntPerThiVil24arX:PureCAlgs}, it follows that $A$ is pure.
\end{proof}

%==========================================================================================
\begin{rmk}
%\label{prp:CompExtCa}
Let $I$ be a closed ideal in a \ca{} $A$.
By \cref{prp:PureIdlQuot}, if $A$ has $0$-comparison (that is, strict comparison of positive elements by quasitraces), then so do $I$ and $A/I$.
Similarly, if $A$ is $0$-almost divisible, then so are $I$ and $A/I$.

In the converse direction, applying \cref{prp:CompExt} and \cref{prp:DivExt}, and arguing as in the proof of \cref{prp:PureExt}, we get:
\begin{enumerate}
\item
If $I$ and $A/I$ have $0$-comparison, then $A$ has $1$-comparison.
\item
If $I$ and $A/I$ are $0$-almost divisible, then $A$ is $1$-almost divisible.
\end{enumerate}
\end{rmk}

%==========================================================================================
Since the reduction phenomenon \cite[Theorem~5.7]{AntPerThiVil24arX:PureCAlgs} does not hold for comparison or divisibility individually, the following questions remain open:

%==========================================================================================
\begin{qst}
\label{qst:CompExt}
Does strict comparison (of positive elements by quasitraces) pass to extensions of \ca{s}?
\end{qst}

%==========================================================================================
\begin{qst}
\label{qst:DivExt}
Does almost divisibility pass to extensions of \ca{s}?
\end{qst}

%==========================================================================================
We end with an application of the main result to the structure of stable multiplier algebras.
A \ca{} is said to be \emph{separably $\calZ$-stable} if it contains a club of separable, $\calZ$-stable sub-\ca{s}.
Since every $\calZ$-stable \ca{} is pure (\cite{Ror04StableRealRankZ}), it follows from \cref{prp:PureSepDet} that every separably $\calZ$-stable \ca{} is pure.
In \cite[Theorem~B]{FarSza24arX:CoronaSSA}, Farah and Szab\'{o} show that a $\sigma$-unital \ca{} $A$ is separably $\calZ$-stable if and only if its multiplier algebra is.
In particular, if $A$ is a unital $\calZ$-stable \ca{}, then its \emph{stable multiplier algebra} $M(A\otimes\Cpct)$ is pure.
In \cref{exa:GroupCa} we give examples of pure, non-$\calZ$-stable \ca{s} whose stable multiplier algebra is pure.

In \cite{KucNgPer10PICorona} and \cite{KafNgZha19PICorona}, the authors study when the corona algebra of a simple \ca{} is purely infinite. Building on these results, we obtain:

%In \cite{KafNgZha19PICorona}, Kaftal, Ng and Zhang study when the corona algebra of a simple \ca{} is purely infinite. 
%Building on their result, we obtain:
%
%==========================================================================================
%\begin{prp}
%\label{prp:PureStableMult}
%Let $A$ be a simple, separable, unital \ca{} of stable rank one and with strict comparison of positive elements by a unique tracial state.
%Then $M(A\otimes\Cpct)$ is pure.
%\end{prp}
%\begin{proof}
%We consider the extension
%\[
%0 \to A\otimes\Cpct \to M(A\otimes\Cpct) \to M(A\otimes\Cpct)/(A\otimes\Cpct) \to 0.
%\]
%The result follows from \cref{prp:PureExt} once we show that $A\otimes\Cpct$ and the stable corona algebra $M(A\otimes\Cpct)/(A\otimes\Cpct)$ are pure.
%
%By \cite[Corollary~8.12]{Thi20RksOps}, $A\otimes\Cpct$ is pure. 
%Further, it follows from \cite[Corollary~4.11]{KafNgZha19PICorona} that $A\otimes\Cpct$ has projection surjectivity and injectivity, which allows us to apply \cite[Corollary~6.12]{KafNgZha19PICorona}.
%Since $A$ has a unique tracial state, it follows that its stable corona algebra is purely infinite, and therefore pure.
%\end{proof}

\begin{prp}
	\label{prp:PureStableMult}
	Let $A$ be a simple, separable, unital \ca.
	\begin{enumerate}
		\item[{\rm (1)}] If $A$ is exact and has finitely many extremal traces and strict comparison of positive elements, then $M(A\otimes\Cpct)$ is pure if and only if $A$ is pure.
		\item[{\rm (2)}] If $A$ has stable rank one and strict comparison of positive elements by a unique tracial state, then $M(A\otimes\Cpct)$ is pure.
	\end{enumerate}
\end{prp}
\begin{proof}
	We consider the extension
	\[
	0 \to A\otimes\Cpct \to M(A\otimes\Cpct) \to M(A\otimes\Cpct)/(A\otimes\Cpct) \to 0.
	\]
	(1): It was proved in \cite[Theorem~4.5]{KucNgPer10PICorona} that, under these assumptions, the corona $M(A\otimes\Cpct)/(A\otimes\Cpct)$ is purely infinite, hence pure. The result then follows from \cref{prp:PureExt}.
	
\noindent (2):	By \cite[Corollary~8.12]{Thi20RksOps}, $A\otimes\Cpct$ is pure. 
	Further, it follows from \cite[Corollary~4.11]{KafNgZha19PICorona} that $A\otimes\Cpct$ has projection surjectivity and injectivity, which allows us to apply \cite[Corollary~6.12]{KafNgZha19PICorona}.
	Since $A$ has a unique tracial state, it follows that its stable corona algebra is purely infinite, and therefore pure. Again the result follows from \cref{prp:PureExt}.
\end{proof}

%==========================================================================================
\begin{exa}
\label{exa:GroupCa}
Let $G$ be a countable, discrete group that is acylindrically hyperbolic with trivial finite radical, and that has the rapid decay property.
Then the reduced group \ca{} $C^*_{\mathrm{red}}(G)$ is simple, separable, unital and has stable rank one (\cite[Theorem~1.1]{GerOsi20InvGpCAlgAcylHyperbolic}, see also \cite{Rau24arX:TwAcylHypGps}), and has strict comparison with respect to its unique tracial state (\cite[Theorem~B]{AmrGaoElaPat24arX:StrCompRedGpCAlgs}).
Therefore, \cref{prp:PureStableMult} applies and shows that $M(C^*_{\mathrm{red}}(G)\otimes\Cpct)$ is pure.

This applies in particular to free groups $\mathbb{F}_n$ for $n=2,3,\ldots,\infty$.
Thus, the stable multiplier algebras of the reduced free group \ca{s} $C^*_{\mathrm{red}}(\mathbb{F}_n)$ are pure.
\end{exa}

%==========================================================================================
%==========================================================================================

%\bibliographystyle{../../aomalphaMyShort}
%\bibliography{../../References}

\providecommand{\etalchar}[1]{$^{#1}$}
\providecommand{\bysame}{\leavevmode\hbox to3em{\hrulefill}\thinspace}
\providecommand{\noopsort}[1]{}
\providecommand{\mr}[1]{\href{http://www.ams.org/mathscinet-getitem?mr=#1}{MR~#1}}
\providecommand{\zbl}[1]{\href{http://www.zentralblatt-math.org/zmath/en/search/?q=an:#1}{Zbl~#1}}
\providecommand{\jfm}[1]{\href{http://www.emis.de/cgi-bin/JFM-item?#1}{JFM~#1}}
\providecommand{\arxiv}[1]{\href{http://www.arxiv.org/abs/#1}{arXiv~#1}}
\providecommand{\doi}[1]{\url{http://dx.doi.org/#1}}
\providecommand{\MR}{\relax\ifhmode\unskip\space\fi MR }
\providecommand{\MRhref}[2]{%
  \href{http://www.ams.org/mathscinet-getitem?mr=#1}{#2}
}
\providecommand{\href}[2]{#2}

\end{document}